\documentclass[12pt]{amsart}
\usepackage{amsmath}
\usepackage{amssymb}

\newtheorem{theorem}{Theorem}[section]
\newtheorem{lemma}[theorem]{Lemma}

\newtheorem{corollary}[theorem]{Corollary}

\theoremstyle{definition}

\newtheorem{question}[theorem]{Question}

\theoremstyle{remark}
\newtheorem{remark}[theorem]{Remark}

\begin{document}

\title[Splitting theorem]
{Two generalizations of Cheeger-Gromoll splitting theorem \\ via
Bakry-Emery Ricci curvature}
\author{Fuquan Fang}
\thanks{The research of the first author was supported by
NSF Grant 19925104 of China, 973 project of Foundation Science of
China, and the Capital Normal University.}
\address{Nankai Institute of Mathematics,
Weijin Road 94, Tianjin 300071, P.R.China}
\address{Department of Mathematics, Capital Normal University,
Beijing, P.R.China}
  \email{ffang@nankai.edu.cn}
\author{Xiang-dong Li}
\thanks{The research of the second author was partially supported by a Delegation in CNRS (2005-2006) at
the Universit\'e Paris-Sud.}
\address{Institut de Math\'ematiques,
Universit\'e Paul Sabatier, 118, route de Narbonne, 31062,
Toulouse Cedex 9, France}
 \email{xiang@math.ups-tlse.fr}
\author{Zhenlei Zhang}
\address{Nankai Institute of Mathematics,
Weijin Road 94, Tianjin 300071, P.R.China}

\date{}

\maketitle

\begin{abstract}
In this paper, we prove two generalized versions of the
Cheeger-Gromoll splitting theorem via the non-negativity of the
Bakry-\'Emery Ricci curavture on complete Riemannian manifolds.
\end{abstract}

\vskip1cm

\section{Introduction}

Cheeger and Gromoll's splitting theorem \cite{CG} played an
important role in the study of manifolds with nonnegative or
almost nonnegative Ricci curvature. In this paper, we consider the
manifolds with nonnegative Bakry-\'{E}mery Ricci curvature and
prove two generalized versions of the splitting theorem on such
manifolds.

Following Bakry-\'Emery \cite{BE}, see also \cite{Q, BQ3, Lo1,
LD}, given a Riemannian manifold $(M,g)$ and a $C^2$-smooth
function $\phi$, $M$ is said to have nonnegative
$\infty$-dimensional Bakry-\'{E}mery Ricci curvature associated to
$\phi$ if $Ric+Hess(\phi)\geq0$, where $Ric$ denotes the Ricci
curvature of $g$ and $Hess$ denotes the Hessian with respect to
$g$. As pointed out in Lott \cite{Lo1}, in general, the splitting
theorem does not hold for manifolds with nonnegative
$\infty$-dimensional Bakry-\'{E}mery Ricci curvature. A trivial
counterexample is given by the hyperbolic $n$-space form
$\mathbb{H}^{n}$, where $Ric+\frac{1}{\delta}Hess(\rho^{2})\geq0$
for some small constant $\delta>0$ and distance function $\rho$.
Obviously there are many lines in this space but it doesn't split
off a line. See \cite{WW}.

If the manifold $M$ is compact and its $\infty$-dimensional
Bakry-\'{E}mery Ricci curvature is positive, then $\pi_{1}(M)$ is
finite. This was first proved by X.-M. Li \cite{LM}, see also
\cite{FG, WW, W, Z}. Also, from Lott's work \cite[Theorem 1]{Lo1},
a compact manifold with nonnegative $\infty$-dimensional
Bakry-\'{E}mery Ricci curvature has $b_{1}$ parallel vector fields
where $b_1$ is the first Betti number of $M$, which are orthogonal
to the gradient field of $\phi$. This indicates that the universal
Riemannian covering space of $(M,g)$ should split off $b_{1}$
lines. We confirm this in this paper, as a corollary of the
following theorem.

\begin{theorem}\label{t01}
Let $(M,g)$ be a complete connected Riemannian manifold with
$Ric+Hess(\phi)\geq0$ for some $\phi\in C^2(M)$ which is bounded
above uniformly on $M$. Then it splits isometrically as
$N\times\mathbb{R}^{l}$, where $N$ is some complete Riemannian
manifold without lines and $\mathbb{R}^{l}$ is the $l$-Euclidean
space. Furthermore, the function $\phi$ is constant on each
$\mathbb{R}^{l}$ in this splitting.
\end{theorem}

Then the corollary reads as
\begin{corollary}\label{c00}
Let $(M,g)$ be a closed connected Riemannian manifold with
$Ric+Hess(\phi)\geq0$ for some smooth function $\phi$ on $M$. Then
we have an isometric decomposition for its universal Riemannian
covering space: $\widetilde{M}\cong N\times\mathbb{R}^{l}$, where
$N$ is a closed manifold, $\mathbb{R}^{l}$ is the $l$-Euclidean
space and $l\geq b_{1}$, the first Betti number of $M$.
Furthermore, the lifting function of $\phi$, say $\tilde{\phi}$,
is constant on each $\mathbb{R}^{l}$-factor.

If $b_{1}$ equals the dimension of $M$, then $(M,g)$ is the flat
torus.
\end{corollary}

Another generalized version of splitting theorem can be described
as follows. According to \cite{Ba94, BQ3}, we say that the
symmetric diffusion operator $L=\Delta-\nabla\phi\cdot\nabla$
satisfies the curvature-dimension condition $CD(0, m)$ if
$$
L|\nabla u|^2\geq {2|L u|^2\over m}+2<\nabla u, \nabla Lu>, \ \
\forall \ u\in C_0^\infty(M).$$ Following the notation used in
\cite{LD}, which is slightly different from \cite{Ba94, BQ3, Lo1},
we define the $m$-dimensional Bakry-\'Emery Ricci curvature of
$L=\Delta-\nabla\phi\cdot\nabla$ on an $n$-dimensional Riemaniann
manifold as follows
$$Ric_{m, n}(L):=Ric+Hess(\phi)-{\nabla\phi\otimes\nabla \phi\over
m-n},$$ where $m={\rm dim}_{BE}(L)>n$ is called the Bakry-\'Emery
dimension of $L$, which is a constant and is not necessarily to be
an integer. By \cite{Ba94, BQ3, LD}, we know that $CD(0, m)$ holds
if and only if $Ric_{m, n}(L)\geq 0$. We now state the following

\begin{theorem}\label{t02}
Let $(M,g)$ be a complete connected Riemannian $n$-manifold and
$\phi\in C^2(M)$ be a function satisfying that $Ric_{m, n}(L)\geq
0$ for some constant $m={\rm dim}_{BE}(L)>n$ which is not
necessarily to be an integer. Then $M$ splits isometrically as
$N\times\mathbb{R}^{l}$ for some complete Riemannian manifold $N$
without line
 and the $l$-Euclidean space $\mathbb{R}^{l}$.
Furthermore, the function $\phi$ is constant on each
$\mathbb{R}^{l}$-factor, and $N$ has non-negative
$(m-l)$-dimensional Bakry-\'Emery Ricci curavture.
\end{theorem}

Our paper provides us with two extensions of the Cheeger-Gromoll
splitting theorem on complete Riemannian manifolds via the
Bakry-\'Emery Ricci curvature. We would like to mention that a
very relevant and independent paper by Wei and Wylie \cite{WW} has
been posted recently in the arxiv. One of their results (see
Theorem $1.4$ in \cite{WW}) says that if $M$ is an $n$-dimensional
complete Riemannian manifold with $Ric+Hess(f)\geq 0$ for some
bounded function $f$ and contains a line, then $M$ splits into
$M=N^{n-1}\times {\mathbb R}$ and $f$ is constant along the line.

The paper is organized as follows: In Section 2, we show that the
Buseman function associated to the line has parallel gradient
field. Then we prove Theorems \ref{t01}, \ref{t02} and Corollary
\ref{c00} in Section 3. In Section 4, we give some remarks on the
Bakry-\'Emery Ricci curvature and the Cheeger-Gromoll splitting
theorem.

\medskip

\noindent {\bf Acknowledgement:} The research was initiated during
the first author's visit to IHES in 2005 and the second author's
visit to the Universit\'e Paris-Sud under a support of
D\'el\'egation au CNRS (2005-2006). The very first version of the
paper was written in 2006.  For some reason, we have not tried to
work out quickly this paper for a submission. In December 2005,
the second author reported Theorem $1.3$ and Theorem \ref{t03} in
the First Sino-French Conference in Mathematics organized by
Zhongshan University at Zhuhai. He would like to thank Professors
D. Bakry, G. Besson, J.-P. Bourguignon, B.-L. Chen, D. Elworthy
and X.-P. Zhu for their interest and helpful discussions.

\section{Estimation of the Laplacian on Buseman function}

 So far, there have been at least three different proofs of the
 Cheeger-Gromoll
splitting theorem. All these proofs turn to prove that the
Busemann function is harmonic. The original proof of Cheeger and
Gromoll \cite{CG} uses the Jacobi fields theory and the elliptic
regularity. The second one by Eschenburg-Heintze \cite{EH} uses
only the Laplacian comparison theorem on distance function and the
Hopf-Calabi maximum principle. The third one, given by Schoen-Yau
\cite{SY}, uses the Laplacian comparison theorem on distance and
the sub-mean value inequality rather than the maximum principle.
For an elegant description of the proof of \cite{EH}, see Besse
\cite{Bs}. We will follow the lines given by \cite{Bs, CG, EH, SY}
to prove the $L$-harmonicity of the Buseman function on $M$.

First we assume $(M,g)$ is a complete Riemannian manifold and
$\phi\in C^{2}(M)$ satisfying that $Ric+Hess(\phi)\geq0$ over $M$.
Fix $p\in M$ as a base point and denote $\rho(x)=\mbox{dist}(p,x)$
the distance function. Given any $q\in M$, let
$\gamma:[0,\rho]\rightarrow M$ be a minimal normal geodesic from
$p$ to $q$ and $\{E_{i}(t)\}_{i=1}^{n-1}$ be parallel orthnormal
vector fields along $\gamma$ which are orthogonal to
$\dot{\gamma}$. Constructing vector fields
$\{X_{i}(t)=\frac{t}{\rho}E_{i}(t)\}_{i=1}^{n-1}$ along $\gamma$
and by the second variation formula, we have the estimate
\begin{eqnarray}
\triangle\rho(q)&\leq&\int_{0}^{\rho}\sum_{i=1}^{n-1}(|\nabla_{\dot{\gamma}}X_{i}|^{2}
-<X_{i},R_{X_{i},\dot{\gamma}}\dot{\gamma}>)dt\nonumber\\
&=&\int_{0}^{\rho}(\frac{n-1}{\rho^{2}}-\frac{t^{2}}{\rho^{2}}Ric(\dot{\gamma},\dot{\gamma}))dt
\nonumber\\
&\leq&\frac{n-1}{\rho}+\int_{0}^{\rho}\frac{t^{2}}{\rho^{2}}Hess(\phi)(\dot{\gamma},\dot{\gamma})dt
\nonumber\\
&=&\frac{n-1}{\rho}+\frac{1}{\rho^{2}}\int_{0}^{\rho}t^{2}\frac{d^{2}}{dt^{2}}(\phi\circ\gamma)
dt\nonumber\\
&=&\frac{n-1}{\rho}+<\nabla\phi,\dot{\gamma}>(q)-\frac{2}{\rho^{2}}\int_{0}^{\rho}t\frac{d}{dt}(\phi\circ\gamma) dt\nonumber\\
&=&\frac{n-1}{\rho}+<\nabla\phi,\dot{\gamma}>(q)-\frac{2}{\rho}\phi(q)
+\frac{2}{\rho^{2}}\int_{0}^{\rho}\phi\circ\gamma dt.\nonumber
\end{eqnarray}
Thus
\begin{eqnarray}\label{e01}
L\rho(q)&\leq&\frac{n-1}{\rho}-\frac{2}{\rho}\phi(q)
+\frac{2}{\rho^{2}}\int_{0}^{\rho}\phi\circ\gamma dt, \ \ \
\forall q\in M\setminus cut(p),
\end{eqnarray}
where $\gamma$ is any minimal normal geodesic connecting $p$ and
$q$.

\begin{lemma}\label{l01}
Let $(M,g)$ be a complete Riemannian manifold and $\gamma$ be a
ray. If $Ric+Hess(\phi)\geq0$ for some smooth function $\phi$
which is bounded from above uniformly on $M$, then the associated
Buesman function of $\gamma$, say $b^{\gamma}$, satisfies that $L
b^{\gamma}\geq0$ in the barrier sense.
\end{lemma}

\begin{remark} We say that a continuous function $f$ on $M$ satisfies
$L f\geq0$ in the barrier sense, if for any given $q\in M$ and
$\epsilon>0$, there is a $C^{2}$ function $f_{q,\epsilon}$ in a
neighborhood of $q$, such that $f_{q,\epsilon}\leq f$ but
$f_{q,\epsilon}(q)=f(q)$, and that $L
f_{q,\epsilon}\geq-\epsilon$. Such $f_{q, \varepsilon}$ is called
a support function of $f$.  We say that $L f\leq0$ in the barrier
sense if $L(-f)\geq0$ in the barrier sense.
\end{remark}

\begin{proof}[Proof of Lemma \ref{l01}] We use the same argument as in \cite{Bs, EH}, see also Lemma 4.7 of
\cite{Z}. Denote $p=\gamma(0)$.  The Buseman function along the
ray $\gamma$ is defined by
$b^{\gamma}(q):=\lim_{t\rightarrow\infty}(t-d(q,\gamma(t)))$. By
\cite{Bs, EH, Z, SY}, $b^\gamma$ is Lipshitz with $1$ as its
Lipshitz constant. Following \cite{Bs, EH, Z}, for any fixed $q\in
M$, we define the support functions around $q$ as follows.

Let $\delta_{t_k}$ be a minimal geodesic connecting $q$ and
$\gamma(t_{k})$. By \cite{Bs, EH, Z}, there exists a subsequence
of $t_k$ such that the initial vector $\dot\delta_{t_k}(0)$
converges to some $X\in T_qM$.  Let $\delta$ be the ray emanating
from $q$ and generated by $X$. Then $q$ does not belong to the
cut-locus of $\delta(r)$ for any $r>0$. So
$b^{\gamma}_{r}(x)=r-d(x,\delta(r))+b^{\gamma}(q)$ is $C^\infty$
around $q$ and satisfies that $b^{\gamma}_{r}\leq b^{\gamma}$ with
$b^{\gamma}_{r}(q)=b^{\gamma}(q)$. On the other hand, by the
estimate (\ref{e01}), we have
\begin{eqnarray*}
L
b^{\gamma}_{r}(x)=-Ld(\delta(r),x)&\geq&\frac{n-1}{d(\delta(r),x)}-\frac{2\phi(x)}{d(\delta(r),x)}\\
& &\
+\frac{2}{d(\delta(r),x)^{2}}\int_{0}^{d(\delta(r),x)}\phi\circ\sigma
dt\nonumber
\end{eqnarray*}
where $\sigma$ is a minimal geodesic connecting $\delta(r)$ and
$x$. Thus for any given $\epsilon>0$, when $r$ is large enough, $L
b^{\gamma}_{r}\geq-\epsilon$ for $x$ in a small neighborhood of
$q$. This shows that $b^{\gamma}_{r}$ is the desired support
function for $b^\gamma$.
\end{proof}

\begin{remark}\label{r02} If $b^{\gamma}$ is smooth at $q$, then $\nabla
b^{+}(q)=\dot\delta(0)$, where $\delta$ is the ray emanating from
$q$ constructed in the proof of Lemma \ref{l01}. See \cite{CG, Z}.
%This can be seen as
%follows: Fix any $s>0$, the points $\delta_{k}(s)$ converge to
%$\delta(s)$, then by definition
%\begin{eqnarray}
%b^{+}(\delta(s))-b^{+}(q)&=&\lim_{k\rightarrow\infty}(d(q,\gamma(t_{k}))-d(\delta(s),\gamma(t_{k})))\nonumber\\
%&=&\lim_{k\rightarrow\infty}(s+d(\delta_{k}(s),\gamma(t_{k}))-d(\delta(s),\gamma(t_{k})))\nonumber\\
%&\geq&\lim_{k\rightarrow\infty}(s-d(\delta_{k}(s),\delta(s)))\nonumber\\
%&=&s.\nonumber
%\end{eqnarray}
%Thus $\frac{d}{ds}|_{s=0}b^{+}(\delta(s))\geq1$ and the result
%follows from the fact that $|\nabla b^{+}|\leq1$.
\end{remark}

\begin{lemma}[The Calabi-Hopf maximum
principal]\label{maxi}  Let $(M, g)$ be a connected complete
Riemaniann manifold, $\phi\in C^2(M)$, and
$L=\Delta-\nabla\phi\cdot\nabla$. Let $f$ be a continuous function
on $M$ such that $Lf\geq 0$ in the barrier sense. Then $f$ attains
no maximum unless it is a constant.
\end{lemma}
\begin{proof} The proof is similar to the one in \cite{C}, see also \cite{Bs}.
\end{proof}

\begin{lemma}\label{l02}
Let $(M,g)$ and $\phi$ be as in Lemma \ref{l01}. Suppose $M$
contains a line $\gamma$, then the Buseman functions $b^{\pm}$
associated to rays $\gamma^{\pm}(t)=\gamma(\pm t),t\geq0,$ are
both smooth and satisfy that $L b^{\pm}=0$.
\end{lemma}
\begin{proof}
By Lemma \ref{l01}, $L(b^{+}+b^{-})\geq0$ in the barrier sense. On
the other hand, $b^{+}+b^{-}=0$ on the line $\gamma$ and the
triangle inequality implies that $b^{+}+b^{-}\leq0$ over $M$. So
$b^{+}+b^{-}\equiv0$ over $M$ by Lemma \ref{maxi}. Now $L
b^{+}\geq0$ and $L(-b^{+})=L b^{-}\geq0$ show that $L b^{+}=0$ in
the barrier sense, then from the elliptic regularity theorem
$b^{+}$ is smooth and $L b^{+}=0$ in the canonical way, cf.
section 6.3-6.4 of \cite{GT}. Similarly $b^{-}$ is smooth
satisfying that $L b^{-}=0$.
\end{proof}

Next we consider the case where $Ric_{m,n}(L)\geq0$. We have the
following

\begin{lemma}\label{l02b}
Let $M$ be a complete Riemannian manifold, $\phi\in C^2(M)$.
Suppose that there exists a constant $m>n$ such that $ Ric_{m,
n}(L)\geq 0.$ Then the Buseman functions $b^{\pm}$ associated to
rays $\gamma^{\pm}(t)=\gamma(\pm t),t\geq0,$ are both smooth and
satisfy that $L b^{\pm}=0$.
\end{lemma}

\begin{proof} Let $b_r^\gamma(x)=t-d(x, \delta(r))+b^\gamma(q)$ be the support
function defined in the proof of Lemma \ref{l01}. By \cite[Remark
3.2]{LD}(pp. 1317-1318), we have the Laplacian comparison theorem
\begin{eqnarray}\nonumber
\left. Ld(\cdot, x)\right|_{y}\leq {m-1\over d(y, x)}, \ \ \forall
\ x\in M, \ y\in M\setminus cut(x).
\end{eqnarray}
This yields that
$$Lb^\gamma_r(x)=-Ld(x, \delta(r))\geq -{m-1\over
d(x, \delta(r))}.$$ Hence, $Lb^{+}\geq 0$ holds in the barrier
sense. Similarly, $Lb^{-}\geq 0$ holds in the barrier sense. By
the same argument as used in the proof of Lemma \ref{l02}, we can
conclude the result. Below we follow \cite{SY} to give an
alternative proof of Lemma \ref{l02b}. Indeed, for all $\psi\in
C_0^\infty(M)$ with $\psi\geq 0$, and for all $t>0$,
\begin{eqnarray*}
\int_M Lb_t^{\gamma}\psi d\mu&=&-\int_M Ld(x, \gamma(t))\psi d\mu\\
&\geq& -\int_M {m-1\over d(x, \gamma(t))}\psi d\mu.
\end{eqnarray*}
Taking $t\rightarrow \infty$, we have $Lb^{+}\geq 0$ in the sense
of distribution. Similarly, $Lb^{-}\geq 0$. Hence
$L(b^{+}+b^{-})\geq 0$ holds in the sense of distribution. By the
strong maximum principle, since the $L$-subharmonic function
$b^{+}+b^{-}$ has an interior maximum on the geodesic ray
$\gamma$, it must be identically constant. Thus, $b^++b^{-}=0$,
$Lb^{\pm}=0$ and $b^{\pm}$ are smooth.
\end{proof}

\begin{lemma}\label{l03}
Under the conditions as in Lemma \ref{l02} or Lemma \ref{l02b},
$\nabla b^{+}$ and $\nabla b^{-}$ are unit parallel vector fields.
\end{lemma}
\begin{proof}
By Remark \ref{r02}, $\nabla b^{\pm}$ are normal vector fields. To
show they are parallel, we will use a generalized version of the
Bochner-Weitzenb\"ock formula. By Bakry-Emery \cite{BE}, for any
smooth function $\psi$, we have
\begin{equation}\label{e03}
L|\nabla\psi|^{2}=2|\nabla^{2}\psi|^{2}+2<\nabla
L\psi,\nabla\psi>+2(Ric+Hess(\phi))(\nabla\psi,\nabla\psi).
\end{equation}
%$$L|\nabla\psi|^{2}=2|\nabla^{2}\psi|^{2}+2<\nabla\triangle\psi,\psi>+2Ric(\nabla\psi,\nabla\psi)
%-2Hess(\psi)(\nabla\phi,\nabla\psi),$$
%$$<\nabla L\psi,\nabla\psi>=<\nabla\triangle\psi,\nabla\psi>-Hess(\phi)(\nabla\psi,\nabla\psi)
%-Hess(\psi)(\nabla\phi,\nabla\psi).$$
Using Lemma \ref{l02} or Lemma \ref{l02b} and applying (\ref{e03})
to $\psi=b^{\pm}$, we see that $0=L|\nabla
b^{\pm}|^{2}\geq2|\nabla^{2} b^{\pm}|^{2}$ over $M$, since
$Ric+Hess(\phi)\geq0$ in both cases. Now the result follows.
\end{proof}

\section{Proof of Theorems \ref{t01} and \ref{t02}}

Now we are in a position to give a
\begin{proof}[Proof of Theorems \ref{t01} and \ref{t02}]
By Lemma \ref{l03}, $X=\nabla b^{+}$ is a parallel unit vector
field. Let $\phi(t)=e^{tX}$ be the one-parameter transformation
group of isometries generated by $X$. The level surface
$N=\{x|b^{+}(x)=0\}$ is a totally geodesic submanifold of $M$, and
the induced metric $h_{N}$ from $g$ is complete. Define a map
$F:N\times\mathbb{R}\rightarrow M$ by
$$F(p,t)=\phi(t)(p).$$ We have
$\frac{d}{dt}b^{+}(\phi(t)p)=|\nabla b^{+}|^{2}(\phi(t)p)\equiv1$.
This implies $F(N,t)\subset\{x\in M|b^{+}(x)=t\}$. We claim that
$F$ is bijective. In fact, for any $x\in M$, letting $q\in N$ be
the nearest point to $x$ and $\gamma$ be the shortest normal
geodesic from $q$ to $x$, then $\dot{\gamma}(q)=X(q)$ and
$\gamma(t)=\phi(t)q$ by the uniqueness of the geodesic, as
$\phi(t)q$ is obviously a normal geodesic. So
$x\in\gamma\subset{\rm Im}(F)$. This proves that $F$ is
surjective. By the group property $F(\cdot,t)\circ
F(\cdot,s)=F(\cdot,t+s)$, $F$ is injective. The claim follows.

Next we prove that $F$ is an isometry. To do so,  notice that
$F(\cdot,t)$ maps $N$ isometrically onto $\{x\in M|b^{+}(x)=t\}$
via $\phi(t)$. So it suffices to show that for any vector $v\in
TN$, we have
$$<dF(\cdot,t)(v),dF(\frac{\partial}{\partial
t})>=<d\phi(t)(v),X>\equiv0.$$ This is obviously true since
$d\phi(t)(TN)\perp X$. So $F$ is an isometry.

Now identifying $(M,g)$ with $(N\times\mathbb{R},h_{N}\otimes
dt^{2})$ and applying (\ref{e03}) to $\psi=b^{+}$, we get
$$0=L|\nabla b^{+}|^{2}=2(Ric+Hess(\phi))(\nabla b^{+},\nabla b^{+})=2\frac{\partial^{2}}{\partial t^{2}}\phi.$$
So $\phi$ is linear on each line of $M$. Since $\phi$ is bounded
from above, it must be constant on each line. This proves Theorem
\ref{t01}.

Finally, if $Ric_{m,n}(L)\geq0$, then (\ref{e03}) yields
$$0=\frac{\partial^{2}}{\partial
t^{2}}\phi\geq\frac{1}{m-n}|\frac{\partial}{\partial
t}\phi|^{2}.$$ So $\phi$ is constant along each line of $M$ and
Theorem \ref{t02} follows.
\end{proof}

\begin{proof}[Proof of Corollary \ref{c00}]
By Theorem 1.1 and using the same argument as did by Cheeger and
Gromoll in \cite{CG}, we conclude the result.
\end{proof}

\begin{remark}
All the arguments in the proof of Theorem \ref{t01} depend only on
the fact that the limit
$$\varlimsup_{\rho\rightarrow\infty}\frac{1}{\rho^{2}}\int_{0}^{\rho}\phi\circ\sigma
dt\leq0
$$ on any ray $\sigma$, see Estimate (\ref{e01}). If so, then Theorem \ref{t01} remains
hold. In particular, if $\frac{\phi(q)}{d(p,q)}=o(1)$ as
$\frac{1}{d(p,q)}\rightarrow0$, where $p$ is a fixed base point,
then Theorem \ref{t01} and all corollaries considered above still
hold.
\end{remark}

Finally we state an alternative result about the splitting
theorem, where the boundedness of the potential function $\phi$ is
removed.
\begin{corollary}\label{c06}
Let $(M,g)$ be an open complete connected Riemannian manifold and
$X$ be a unit parallel vector field. If $Ric+Hess(\phi)\geq cg$
for some $\phi\in C^{2}(M)$ and a constant $c>0$, then $(M,g)$
splits off a line. In particular, any open shrinking Ricci soliton
with a parallel vector field splits off a line.
\end{corollary}

Recall that a Riemannian manifold $(M,g)$ is a shrinking Ricci
soliton if there exists a smooth function $f$ such that
$Ric+Hess(f)=cg$ for some positive constant $c$.

\begin{proof}
By the result of \cite{WW, W}, such a manifold $M$ has finite
fundamental group $\pi_{1}(M)$. Denote by $(\widetilde{M},
\tilde{g})$ the universal Riemannian covering of $(M, g)$ and let
$\tilde{X}$ be the lifting of $X$. Let $\phi(t)=e^{t\tilde{X}}$
and $N$ be a maximal integral submanifold of $\tilde{X}^\perp$,
the distribution orthogonal to $\tilde{X}$. Define the map $F$ as
in the proof of Theorem \ref{t01} and Theorem \ref{t02}. Then it
can be shown that $F$ is an isometry by the simply connectedness
of $\widetilde{M}$. So we can identify $(\widetilde{M},\tilde{g})$
with $(N\times\mathbb{R},h_{N}\otimes dt^{2})$, where $h_{N}$ is
the restriction of $\tilde{g}$ on $N$. Then the vector field
$\widetilde{X}$ equals ${\partial\over\partial t}$ and is
invariant under the action by $\pi_{1}(M)$. We claim that
$\pi_{1}(M)$ acts trivially on the $\mathbb{R}$-factor.

Suppose not, then there is $\alpha\in\pi_{1}(M)$ and
$(z_{0},t_{0})\in N\times\mathbb{R}$ such that
$\alpha(z_{0},t_{0})=(z_{1},t_{1})$ with $t_{0}\neq t_{1}$. Then
$\alpha$ must maps the line $\{z_{0}\}\times\mathbb{R}$
isometrically onto the line $\{z_{1}\}\times\mathbb{R}$ in the
same direction and maps the slice $N\times\{t_{0}\}$ onto the
slice $N\times\{t_{1}\}$, because it preserves
$\frac{\partial}{\partial t}$. Denote by
$p:N\times\mathbb{R}\rightarrow\mathbb{R}$ the projection to the
$\mathbb{R}$-factor and $i:\mathbb{R}\rightarrow
N\times\mathbb{R}$ the injection $i(t)=(z_{0},t)$. Then
$\bar{\alpha}=p\circ \alpha\circ i$ is a translation on
$\mathbb{R}$ with variation $t_{1}-t_{0}\neq0$. Now
$\{\bar{\alpha}^{k}=p\circ \alpha^{k}\circ i\}_{k=1}^{\infty}$
forms a subgroup of isometry transformation group of $\mathbb{R}$,
which is generated by a translation. This shows that $\alpha$ is a
free element of $\pi_{1}(M)$, which contradicts the finiteness of
$\pi_{1}(M)$. Hence $\pi_{1}(M)$ acts trivially on
$\mathbb{R}$-factor and consequently the base manifold $(M, g)$
splits off a line.
\end{proof}

It is natural to pose the following questions.

\begin{question} Construct a compact Riemannian manifold with negative Ricci curvature somewhere and
with positive Bakry-\'Emery-Ricci curvature everywhere.
\end{question}

\begin{question}
Let $(M,g)$ be an open complete Riemannian manifold with
$Ric+Hess(\phi)\geq cg$ for some function $\phi\in C^2(M)$ and
some constant $c\in {\mathbb R}$. If $(M,g)$ contains a line, does
it really split off a line? In particular, is it true on a
shrinking Ricci soliton?
\end{question}

\section{Some remarks}

In this section, we give some remarks on  the Bakry-\'Emery-Ricci
curvature and the Cheeger-Gromoll splitting theorem.

From our personal conversation with Dominique Bakry, we are able
to know some interesting history about the introduction of the
Bakry-\'Emery Ricci curvature which was named by Lott in
\cite{Lo1}. In the beginning of 1980s, when Bakry studied some
problems of the Riesz transforms associated with diffusion
operators, he observed that some commutation formulae play an
essential role. In \cite{Ba85}, he introduced the so-called
$\Gamma$-operator ({\it le carr\'e du champs}) and the
$\Gamma_2$-operator ({\it le carr\'e du champs it\'er\'e}) which
are symmetric bilinear derivative operators acting on nice
functions. In \cite{BE}, Bakry and \'Emery formulated $\Gamma$ and
$\Gamma_2$ in a very general setting and proved that, if
$\Gamma_2\geq \lambda \Gamma$ for some constant $\lambda>0$, then
the logarithmic Sobolev inequality holds. For a symmetric
diffusion operator $L=\Delta- \nabla\phi\cdot\nabla$ on a
Riemannian manifold, Bakry and \'Emery  \cite{BE} obtained the
fundamental weighted Bochner-Lichnerowicz-Weitzenb\"ock formula,
which says that $Ric+\nabla^2\phi$ is exactly the Ricci curvature
term in the Bochner-Lichnerowciz-Weitzenb\"ock decomposition of
the Witten Laplacian on one-forms on Riemannian manifolds equipped
with weighted volume measures $e^{-\phi}dv$. The notion of the
curvature-dimension $CD(K, n)$-condition for diffusion operators
has also been introduced in \cite{BE}. During 1980-2006, the study
of the comparison theorems of the Bakry-\'Emery-Ricci curvature
has been extensively developed by Bakry and his collaborators
\cite{Ba87, Ba94, BL1, BL2, BQ1, BQ2, BQ3, BL3}. In \cite{BL2},
Bakry and Ledoux proved the generalized Myers theorem for
Markovian diffusion operators satisfying the $CD(R, n)$-condition
with $R>0$. In \cite{Q}, Qian extended the Myers theorem to
complete Riemannian manifolds on which the $(n+\alpha)$-
dimensional Bakry-\'Emery Ricci curvature $Ric+Hess(h)-\alpha^{-1}
\nabla h\otimes \nabla h$ has a uniformly positive lower bound.
In \cite{BQ3}, Bakry and Qian proved a generalized version of the
Bishop-Gomov volume comparison theorem and the generalized
Laplacian comparison theorem. In \cite{LD}, the second author of
this paper gave a new proof of the generalized Laplacian
comparison theorem and extended S.-T. Yau's $L^\infty$-Liouville
theorem, P. Li's $L^1$-Liouville and $L^1$-uniqueness theorem to
the elliptic equation or the heat equation associated with
symmetric diffusion operators on complete Riemannian manifolds.

In \cite{Lo1}, Lott gave a new understanding of the Bakry-\'Emery
Ricci curvature and obtained some interesting results from the
point of view of the measured Gromov-Hausdroff convergence. In
2002-2003, Perelman posted three papers on Arxiv. In \cite{P},
Perelman used the Bakry-\'Emery Ricci curvature to give a modified
version of R. Hamilton's Ricci flow equation. Under the condition
that the weighted measure $e^{-\phi(x)}\sqrt{\det g(x)}dx$ does
not change, Perelman \cite{P} proved that the equation of the
modified Ricci flow for the Riemannian metric $g$ can be viewed as
the gradient flow of the ${\mathcal F}$-functional introduced in
\cite{P} and that the potential function $\phi$ satisfies a
conjugate backward heat equation. He then introduced the so-called
${\mathcal W}$-entropy functional and proved that the ${\mathcal
W}$-functional is monotonically decreasing along the modified
Ricci flow. By the monotonicity of the ${\mathcal W}$-functional
and using an earlier result due to Rothauss on L. Gross'
logarithmic Sobolev inequalities, Perelman \cite{P} proved the
Little Loop Lemma which was conjectured by R. Hamilton in the
1990s. This result consists of one of the most important steps in
Perelman's final resolution of the Poincar\'e conjecture.

By the Bishop-Gromov volume comparison theorem for the weighted
volume measure and using the standard argument as used in Gromov's
original proof \cite{G1, G2} for his famous theorem, we can extend
the Gromov precompactness theorem to compact Riemannian manifolds
with weighted measures via the finite dimensional Bakry-\'Emery
Ricci curavture. More precisely, we have the following

\begin{theorem}\label{t03} Let ${\mathcal M}(m, n, d, K)$ be the set of $n$-dimensional compact Riemannian
manifolds $(M, g)$ equipped with $C^2$-weighted volume measures
$d\mu=e^{-\phi}dv$ such that: $n\leq {\rm dim}_{BE}(L)\leq m$,
${\rm diam}(M)\leq d,$ and $Ric_{{\rm dim}_{BE}(L), n}(L)\geq K$,
where ${\rm dim}_{BE}(L)$ is the Bakry-Emery dimension of the
diffusion operator $L=\Delta-\nabla\phi\cdot\nabla$. Then
${\mathcal M}(m, n, d, K)$ is precompact in the sense of the
measured Gromov-Hausdroff convergence.
\end{theorem}

To our knowledge, at least in the case ${\rm dim}_{BE}(L)=m$ and
$\phi\in C^\infty(M)$, Theorem \ref{t03} has been already pointed
out by Lott \cite[Remark 3, p. 881]{Lo1}. Indeed, if
$L=\Delta-\nabla\phi\cdot\nabla$ is a symmetric diffusion operator
with $Ric_{m, n}(L)\geq K$ for some $m\geq n$ and $K\in {\mathbb
R}$, then it is obviously true that $Ric_{m', n}(L)\geq K$ for all
$m'\geq m$. So, if ${\rm dim}_{BE}(L)\leq m$ and if $Ric_{{\rm
dim}_{BE}(L), n}(L)\geq K$, then obviously we have $Ric_{m,
n}(L)\geq K$. Therefore, Theorem \ref{t03} can be recaptured from
the above mentioned result due to Lott \cite{Lo1}, which holds
obviously when $\phi\in C^2(M)$.

Theorem \ref{t02} and Theorem \ref{t03} were obtained in November
2005. During December 12-18, 2005, the second author of this paper
reported Theorem \ref{t02} and Theorem \ref{t03} in the First
Sino-French Conference of Mathematics organized by Zhongshan
University in Zhuhai. During this conference, the second author
had a very nice discussion with Prof. J.-P. Bourguignon on the
Gromov precompactness theorem on weighted Riemannian manifolds.
Prof. Bourguignon also told him that the recent study of M.
Kontsevich on the Conformal Fields Theory needs some kind of
generalizations of the Gromov precompactness theorem.

To what extent can we say about the geometry on the measured
Gromov-Hausdroff limit of the weighted Riemannian manifolds in
${\mathcal M}(m, n, d, K)$? This problem has been a central issue
in the study of the Riemannian geometry in the Large. In
2005-2006, Lott-Villali \cite{LV} and  Sturm \cite{St1, St2}
independently developed a Comparison Geometry of the Ricci
curvature on Metric-Measure Spaces. Roughly speaking, they used
the optimal transports and the entropy functions to introduce the
notions of infinite dimensional Ricci curvature and the finite
dimensional Ricci curvature $Ric_N$ as well as the
curvature-dimension ${\rm\bf CD}(K, N)$-condition on
metric-measure spaces. In \cite{LV}, Lott and Villali proved that,
if $(X, d, \mu)$ is the measured Gromov-Hausdroff limit of a
sequence of compact metric-measure spaces $(X_i, d_i, \mu_i)$, and
if the generalized $N$-dimensional Ricci curvature $Ric_N$ (resp.
the generalized $\infty$-dimensional Ricci curvature) on each
$(X_i, d_i, \mu_i)$ is non-negative, $N\in [1, \infty)$, then $(X,
d, \mu)$ has non-negative generalized $N$-dimensional Ricci
curvature (resp, generalized $\infty$-dimensional Ricci
curvature). Similar stability theorem was also proved by Sturm
\cite{St2} under the ${\rm\bf CD}(K, N)$-condition. It might be
interesting to point out that there are some differences between
Lott-Villani's definition of $Ric_N$ and Sturm's definition of the
${\rm\bf CD}(K, N)$-condition on metric-measure spaces. In the
case of Riemannian manifolds, their definitions coincide with the
standard ones in Riemannian geometry. However, in general they are
not the same. Under the condition that $Ric_{N}\geq 0$ and
$Ric_{\infty}\geq K$ for some $K>0$ on a compact metric-measure
space $(X, d, \nu)$, Lott and Villali proved the so-called Weak
Bonnet-Myers theorem, i.e., ${\rm diam}(X)\leq C\sqrt{N\over K}$
for some universal constant $C>0$. See Theorem 6.30 in \cite{LV}.
On the other hand, under the ${\bf\rm CD}(K, N)$-condition, where
$K>0$ and $N\geq 1$ are two constants, Sturm \cite{St1, St2}
proved that ${\rm diam}(X) \leq \pi \sqrt{N-1\over K}$, which is
the same upper bound for the diameter as indicated in the
classical Bonnet-Myers theorem on complete $N$-dimensional
Riemannian manifolds with Ricci curvature $Ric\geq K>0$.

In \cite{Lo2}, Lott pointed out that the Cheeger-Gromoll splitting
theorem cannot be extended to metric-measure spaces with
non-negative finite dimensional Ricci curvature (in the sense of
Lott-Villali \cite{LV}). Whether or not the Cheeger-Gromoll
splitting theorem can be extended to metric-measure spaces
satisfying the curvature-dimension ${\rm\bf CD}(0, N)$ condition
in the sense of Sturm \cite{St2}, or how to introduce a more
reasonable definition of the finite dimensional Ricci curvature on
metric-measure spaces so that the Cheeger-Gromoll splitting
theorem can be extended? This might be an interesting problem for
a study in future.

\nocite{*}

\end{document}